\documentclass[12pt]{article}
\usepackage{amssymb}

\usepackage{amsmath,amsfonts,amssymb}

\newtheorem{defn}{Definition}[section]
\newtheorem{theo}[defn]{Theorem}
\newtheorem{lem}[defn]{Lemma}
\newtheorem{prop}[defn]{Proposition}

\newtheorem{rem}[defn]{Remark}

\newenvironment{proof}{{\bf Proof }}{{\vskip 0.1cm \hfill$\Box$}}

\def\esssup{\text{ess\,sup}}  

\begin{document} 

\noindent
{\Large\bf Existence and approximation of Hunt processes associated with generalized Dirichlet forms}{\footnote{Supported by the research projects 0450-20100007 and 0450-20110022 of Seoul National University, 
the CRC 701 and the BIBOS-Research Center at Bielefeld University.}\\[10mm]
\noindent {\small{{\bf Vitali Peil},Universit\"at Bielefeld, Fakult\"at f\"ur Mathematik, Postfach 100131, 33501 Bielefeld,
Germany, E-mail: vitali.peil@uni-bielefeld.de, tel.: +49 521 106-6125}}\\[3mm]
\noindent {\small{{\bf Gerald Trutnau}, Seoul National University, Department of Mathematical Sciences and Research Institute of Mathematics,
599 Gwanak-Ro, Gwanak-Gu, Seoul 151-747, South Korea, E-mail: trutnau@snu.ac.kr, tel.: +82 0 2 880 2629; fax: +82 0 2 887 4694.}}\\[9mm]
\noindent
{\small{\bf Abstract.} We show that any strictly quasi-regular generalized Dirichlet form that satisfies the mild structural 
condition {\bf D3} is associated to a Hunt process, and that the associated Hunt process can be approximated  by a sequence of  multivariate Poisson processes. 
This also gives a new proof for the existence of a Hunt process associated to a strictly quasi-regular generalized Dirichlet form that satisfies {\bf SD3} and extends all 
previous results.  \\[3mm]
\section{Introduction}\label{I}
In this note we are concerned with several questions related to probabilistic and analytic potential theory of generalized Dirichlet forms.
A particular aim is to find definitive analytic conditions for non-sectorial Dirichlet forms that ensure the existence of an associated Hunt process.
The question whether the associated process is a Hunt process is crucial for localizing purposes (see e.g. introduction of Ref. \cite{Tr5}).\\
A fundamental consequence of Theorem 2 in Ref. \cite{Tr5} and Theorem 3.2(ii) in Ref. \cite{RoTr} is that
any {\it transient} Hunt process $\mathbb{M}$ on a metrizable and 
separable state space is strictly properly associated in the resolvent sense with a strictly quasi-regular generalized Dirichlet form.
This is relevant because we can then apply all the fine results from the potential theory of generalized Dirichlet forms w.r.t. 
the strict capacity (see Ref. \cite{Tr5} for some strict potential theory, 
and Remark 3.3(iv) in Ref. \cite{RoTr} which applies also to strictly quasi-regular 
generalized Dirichlet forms and Hunt processes). 
Moreover, if the state space is only slightly less general, namely (for tightness reasons) a metrizable Lusin space, then by Theorem 2.1 in Ref. \cite{mrs} 
the Hunt process can be approximated by multivariate Poisson processes and the 
approximation works for all $P_x$, i.e. for all $x$ in the state space. The canonical approximation of the Hunt 
process by Markov chains is useful as it provides an additional tool for its analysis and for the analysis of the underlying generalized Dirichlet form. 
Note that the just mentioned line of arguments is not valid for sectorial Dirichlet forms, which underlines a strength of generalized Dirichlet form theory. 
In fact, for a given arbitrary Hunt process we first do not know how to check whether it is associated to a sectorial Dirichlet form, and second this is clearly is not true in general. \\ 
Here, we establish the \lq\lq quasi converse\rq\rq\ of the above with nearly no restriction on the state space. We consider 
two problems, which, due to the method, are in fact solved simultaneously. 
The first problem is to establish the existence of an associated Hunt process to a strictly quasi-regular generalized Dirichlet form on a general state space, 
and the second is the approximation of this Hunt process in a canonical way through Markov chains.  
The second problem goes back to an original idea of S. Ethier and T. Kurtz. 
In fact, it is shown in Chapter 4.2 in Ref. \cite{ek} that for nice state spaces such 
as  locally compact and separable state spaces and nice transition semigroups like Feller ones, the Yosida approximation via multivariate 
Poisson processes converges for all starting points to a Markov process with the given semigroup. This was generalized in Ref. \cite{mrz} 
where it is shown that the Yosida approximation of the generator together with some tightness arguments that result from the strict quasi-regularity 
leads to the approximation via multivariate Poisson processes of any Hunt process that is associated with a strictly quasi-regular sectorial Dirichlet form. 
This also led to a new proof for the existence of an associated Hunt process.
However, the price for the increased generality is that the approximation only works for strictly quasi-every starting point $x$ of the state space. 
Since we use the same method we have to pay the same price, and even more we have to assume the additional structural condition {\bf D3} that is however trivially 
satisfied for any sectorial Dirichlet form (see Proposition \ref{sd3}). Nonetheless, since the class of generalized Dirichlet forms is much larger than the 
class of sectorial Dirichlet forms our results represent a considerable generalization. In particular time-dependent processes and 
processes corresponding to far-reaching perturbations of symmetric (or even sectorial) forms, are covered.\\
Besides the canonical approximation scheme through Markov chains we want to emphazise that our main result 
Theorem \ref{hunt2} is a definitive improvement of Theorem 3 in Ref. \cite{Tr5}. 
Applying here a quite different method than in Ref. \cite{Tr5} we were able to relax the algebra structure 
condition {\bf SD3} of Ref. \cite{Tr5} to the quite weaker linear structure condition {\bf D3}. 
Therefore our general analytic conditions
for non-sectorial Dirichlet forms to ensure the existence of an associated Hunt process are just {\bf D3} and the strict quasi-regularity.  
The state space is only assumed to be a Hausdorff topological space such that its Borel $\sigma$-algebra is 
generated by the set of continuous functions on the state space. Our result is hence the counterpart of IV. Theorem 2.2 in 
Ref. \cite{St1} for Hunt processes.\\
Finally let us briefly summarize the main contents of this paper. Section \ref{2} contains some preliminaries and the fundamental results. 
In particular, our way of defining the strict capacity (cf. Definition \ref{cap}) is more explicit than in V.2 of Ref. \cite{mr} and  Section 2 of Ref. \cite{mrz}, 
but still equivalent (see Remark \ref{all}). 
The strict capacity is defined w.r.t. some reference function $\varphi$, but it turns out to be like the $\varphi$-capacity independent of that function (see Remark \ref{all}(ii)). 
Proposition \ref{a} provides a useful new estimate for the strict capacity.
A crucial result is the construction of the modified functions $\widehat{e}_n$ in Lemma \ref{gf} in comparison to the functions $e_n$ of Lemma 3.5 in Ref. \cite{mrz}.  
This makes the difference and allows to handle the non-sectorial case (see also Remark \ref{enr} for some related explanations).
Lemma \ref{gf} allows to get the important tightness result of Lemma \ref{liq2}. Note that 
we also correct an inaccuracy that appears in the proofs of the statements corresponding to Lemma \ref{liq2} in both 
papers Ref. \cite{mrz} and Ref. \cite{mrs} (see Remark \ref{mis}) and that we partially improve results from Ref. \cite{mrz} (see e.g. paragraph in front of Proposition \ref{sp}).
Having developed the fundamental results of potential theory in Section 2, most of the results of Sections \ref{3} and \ref{4} 
follow by \lq\lq routine\rq\rq\ arguments from Ref. \cite{mr}, Ref. \cite{ek}, and Ref. \cite{get},
similarly to the line of arguments used in Ref. \cite{mrz}. For the sake of completeness and convenience of the reader we summarize these results.
\section{Strict quasi-regularity, strict capacity, and the construction of $R_{\alpha}$}\label{2}
For notations and notions that might not be defined here we refer to Ref. \cite{Tr5} and references therein. Throughout the paper let $E$ be a Hausdorff space such that its Borel $\sigma$-algebra ${\cal{B}}(E)$ is
generated by the set ${\cal{C}}(E)$ of all continuous functions on $E$, and let $m$ be a
$\sigma$-finite measure on $(E,{\cal{B}}(E))$. Let $({\cal{E}},{\cal{F}})$ be a generalized Dirichlet form with sectorial part $({\cal A},{\cal V})$ on ${\cal{H}}=L^2(E,m)$. 
Let $(G_\alpha)_{\alpha>0}$ be $L^2(E,m)$-{\it resolvent associated with} ${\cal{E}}$, and $(\widehat G_\alpha)_{\alpha >0}$ be 
the adjoint of $(G_\alpha)_{\alpha>0}$ in ${\cal{H}}$.
\subsection{Quasi-regular generalized Dirichlet forms and the conditions {\bf D3} and {\bf SD3}}\label{2.1} 
Given $\varphi\in L^2(E,m)$, $\varphi>0$, an increasing sequence of closed subsets $(F_k)_{k\ge 1}$ of $E$
is an ${\cal{E}}$-nest, if
$$
\mbox{Cap}_{\varphi}(F_k^c)=\int_E (G_1 \varphi)_{F_k^c}\varphi\,dm\longrightarrow 0 \ \mbox{ as } k \to \infty.
$$ 
By IV. Proposition 2.10 in Ref. \cite{St1} the notion of ${\cal{E}}$-nest is independent of the special choice of $\varphi$.
Accordingly to $\mbox{Cap}_{\varphi}$, ${\cal{E}}$-exceptional sets, ${\cal{E}}$-quasi-continuity, etc., and the quasi-regularity are defined (see Ref. \cite{St1}). \\  
In contrast to the theory of sectorial Dirichlet forms in Ref. \cite{mr} and Ref. \cite{semi_df} it is not known whether 
quasi-regularity is sufficient for the existence of an associated standard process in case of a generalized Dirichlet form. 
Therefore the following condition  
\begin{itemize}
 \item[$\mathbf{D3}$]There exists a linear subspace $\mathcal{Y}\subset{\cal{H}}\cap L^{\infty}(E,m)$ such that $\mathcal{Y}\cap{\cal{F}}$ is dense in ${\cal{F}}$, 
 $\lim_{\alpha\rightarrow\infty}e_{\alpha G_{\alpha} u-u}=0$ in ${\cal{H}}$ for all $u\in\mathcal{Y}$ and for the closure $\overline{\mathcal{Y}}$ of $\mathcal{Y}$ 
 in $L^{\infty}(E;m)$ it follows that $u\wedge\alpha\in \overline{\mathcal{Y}}$ for $u\in \overline{\mathcal{Y}}$ and $\alpha \geq 0.$
\end{itemize}
is introduced in  IV.2 in Ref. \cite{St1} and it is shown in IV. Theorem 2.2 of Ref. \cite{St1} that a {\it quasi-regular generalized Dirichlet 
form satisfying} {\bf D3} 
is associated with an $m$-{\it tight special standard process}.\\
By an algebra of functions we understand a {\it linear} space that is closed under multiplication. The following condition
\begin{itemize}
\item[$\mathbf{SD3}$] There exists an algebra of functions $\mathcal{G}\subset {\cal{H}} \cap L^{\infty}(E,m)$ such that $\mathcal{G}\cap {\cal{F}}$ is dense in ${\cal{F}}$ and 
$\lim_{\alpha\rightarrow \infty}e_{\alpha G_{\alpha} u-u}$ in ${\cal{H}}$ for every $u\in\mathcal{G}.$
\end{itemize}
was introduced in Ref. \cite{Tr5}. We have the following: 
\begin{prop}\label{sd3}
It holds:
\begin{itemize}
	\item[(i)]  {\bf SD3} implies {\bf D3}.
	\item[(ii)] {\bf SD3} holds for any (sectorial semi-)Dirichlet form $({\cal{E}},D({\cal{E}}))$ on $L^2(E;m)$ with $\mathcal{G}=D({\cal{E}})\cap L^{\infty}(E,m)$.
\end{itemize}
\end{prop}
\begin{proof}
(i) The proof is standard; cf. e.g. proof of IV. Proposition 2.1 in Ref. \cite{St1}.\\
(ii) Follows from I.Corollary 4.15 and Proposition 4.17(ii) in Ref. \cite{mr}. 
\end{proof}

\subsection{Strict capacities and strictly quasi-regular generalized Dirichlet forms}
We fix $\varphi\in L^1(E,m)\cap {\cal{B}}$ with $0<\varphi(z)\leq 1$ for every $z\in E$. The 
following definition is a notational simplification of Definition 1 of Ref. \cite{Tr5}.
\begin{defn}\label{cap}
For $U\subset E$, $U$ open, set 
$$
\mbox{Cap}_{1,\widehat G_1 \varphi}(U):=\int_E e_{U} \varphi\, dm
$$ 
where $e_U:=\lim_{k\rightarrow\infty}(1\wedge G_1(k\varphi))_U$ exists as a bounded and increasing limit in $L^{\infty}(E,m)$. If 
$A\subset E$ arbitrary then 
$\mbox{Cap}_{1,\widehat G_1 \varphi}(A):=\inf\{\mbox{Cap}_{1,\widehat G_1 \varphi}(U)| \,U\supset A, U \mbox{open}\}$. 
\end{defn}
By Theorem 1 of Ref. \cite{Tr5} $\mbox{Cap}_{1,\widehat G_1 \varphi}$ is a finite Choquet capacity. A priori the function $e_U$ depends on the chosen $\varphi$ but in the next lemma we will see that this is actually not the case.
\begin{lem}\label{capeq}
 $U\subset E$, $U$ open.
 \begin{itemize}
\item[(i)] We have
$$
\mbox{Cap}_{1,\widehat G_1 \varphi}(U)=\sup_{u\in {\cal{P}}_{\cal{F}},u\le 1}{\cal{E}}_1(u_U,\widehat G_1 \varphi),
$$
where ${\cal{P}}$ denotes the $1$-excessive elements of ${\cal{V}}$, and ${\cal{P}}_{\cal{F}}=\{u\in {\cal{P}}|\exists f\in{\cal{F}}, u\le f\}$.
\item[(ii)] It holds 
$$
e_U=\esssup\{u_U\,|\,u\in {\cal{P}}_{\cal{F}},u\le 1\}.
$$
\end{itemize}
\end{lem}
\begin{proof}
(i) Clearly \lq\lq$\le$\rq\rq\ holds in the statement. 
For $f\in L^2(E;m)$ it is not difficult to see that $G_1 f>0$ $m$-a.e. whenever $f>0$ $m$-a.e. Indeed, since $m$ is $\sigma$-finite and $(\widehat{G}_{\alpha})_{\alpha>0}$ is 
positivity preserving by I. Remark 4.2 of Ref. \cite{St1}, we can easily show by I. Proposition 3.4 in Ref. \cite{St1} that 
if $A\in {\cal B}(E)$ and $\int_A G_1 f \,dm=0$, then $m(A)=0$. Thus $\sup_{k\ge 1}1\wedge G_1(k \varphi)=1$ $m$-a.e.
But then we have by (9) of Ref. \cite{Tr2}
\begin{eqnarray*}
\sup_{u\in {\cal{P}}_{\cal{F}},u\le 1}{\cal{E}}_1(u_U,\widehat G_1 \varphi) & = & \sup_{u\in {\cal{P}}_{\cal{F}},u\le 1}\sup_{\alpha\ge 1}{\cal{E}}_1(u,(\widehat G_1 \varphi)_U^{\alpha})\\
& \le & \sup_{\alpha\ge 1}\sup_{k\ge 1}{\cal{E}}_1(1\wedge G_1(k \varphi),(\widehat G_1 \varphi)_U^{\alpha})=\mbox{Cap}_{1,\widehat G_1 \varphi}(U).\\
\end{eqnarray*}
(ii) Define ${\cal{P}}_{\cal{F}}^{1}:=\{u\in {\cal{P}}_{\cal{F}}\,|\,u\le 1\}$. Since $e_U=\esssup\{\big (1\wedge G_1(k\varphi)\big )_U\,|\, k\ge 1\}$, it is enough to show that $e_U\ge u_U$ for every $u\in {\cal{P}}_{\cal{F}}^{1}$. Suppose this is wrong. Then there is some $f\in {\cal{P}}_{\cal{F}}^{1}$ such that
\begin{equation}\label{zwei}
\int_E \left (f_U\vee e_{U}\right ) \,\varphi\, dm>\int_E e_{U} \varphi\, dm=\mbox{Cap}_{1,\widehat G_1 \varphi}(U).
\end{equation}
On the other hand since obviously $f_U\vee \big (1\wedge G_1(k\varphi)\big )_U\le \big (f\vee \big (1\wedge G_1(k\varphi)\big )\big )_U\in {\cal{P}}_{\cal{F}}^{1}$ for any $k$ we obtain by (i)
\begin{eqnarray*}
\int_E \left (f_U\vee e_{U}\right ) \,\varphi\, dm & = & 
\lim_{k\rightarrow\infty}\int_E  \big (f_U\vee \big (1\wedge G_1(k\varphi)\big )_U \big ) \,\varphi\, dm\\
& \le & \limsup_{k\rightarrow\infty} \int_E  \big (f\vee \big (1\wedge G_1(k\varphi)\big )\big )_U\,\varphi\, dm\le \mbox{Cap}_{1,\widehat G_1 \varphi}(U),\\
\end{eqnarray*}
which contradics (\ref{zwei}).
\end{proof}
\begin{rem}\label{all}
\begin{itemize}
	\item[(i)] Let $\overline{\mbox{Cap}}_{1,\widehat G_1 \varphi}$ be the strict capacity 
as defined in V.2 of Ref. \cite{mr}. Then in the sectorial case, i.e. when $({\cal E}, {\cal F})$ is a Dirichet form, we have ${\cal{F}}={\cal{V}}$ and  ${\cal{P}}_{\cal{F}}={\cal{P}}$. 
Thus $\mbox{Cap}_{1,\widehat G_1 \varphi}=\overline{\mbox{Cap}}_{1,\widehat G_1 \varphi}$ by Lemma \ref{capeq} and Definition V.2.1 in \cite{mr}. 
Moreover, the function $e_U$ defined in Definition \ref{cap} is an explicit realization of the functon $e_U$ of Lemma 2.2 in Ref. \cite{mr}.
	\item[(ii)] It follows immediately from Lemma \ref{capeq} that the strict ${\cal{E}}$-nests, and hence the strict notions, do not depend on the  special choice of $\varphi$. 
\end{itemize}
\end{rem}
Adjoining the cemetery $\Delta$ to $E$ we let $E_{\Delta}:=E\cup\{\Delta\}$ and ${\cal{B}}(E_{\Delta})={\cal{B}}(E)\cup\{B\cup \{\Delta\}|B\in{\cal{B}}(E)\}$.
 We will consider different topologies on 
$E_{\Delta}$. 
If $E$ is a locally compact separable metric space but not compact, 
$E_{\Delta}$ will be the one point compactification of $E$, i.e. the open sets of $E_{\Delta}$ are the open sets of $E$ 
together with the sets of the form $E_{\Delta}\setminus K$, $K\subset E$, $K$ compact in $E$. Otherwise 
we adjoin the cemetery $\Delta$ to $E$ as an isolated point. 
We extend 
$m$ to $(E_{\Delta},{\cal{B}}(E_{\Delta}))$ by setting $m(\{\Delta\})=0$. 
Any real-valued function u on $E$ is extended to $E_{\Delta}$ by setting $u(\Delta)=0$.  
\\
Given an increasing sequence $(F_k)_{k\ge 1}$ of closed subsets of $E$, we define
$$
C_{l,\infty}(\{F_k\})=\{f:A\rightarrow \mathbb{R}\mid \bigcup_{k\ge 1}F_k\subset A\subset E,\, 
f_{\mid F_k\cup \{\Delta\}} \mbox{ is lower semicontinuous } \forall k\},
$$ 
and $C(\{F_k\})$, $C_{\infty}(\{F_k\})$ as on page 360 of Ref. \cite{Tr5}.\\
Accordingly to $\mbox{Cap}_{1,\widehat G_1 \varphi}$, the notions strictly ${\cal{E}}$-exceptional (s.${\cal{E}}$-exceptional), strict ${\cal{E}}$-nest (s.${\cal{E}}$-nest), 
strictly ${\cal{E}}$-quasi-everywhere (s.${\cal{E}}$-q.e.), strictly ${\cal{E}}$-quasi-continuous (s.${\cal{E}}$-q.c.), and strictly ${\cal{E}}$-quasi-lower-semicontinuous (s.${\cal{E}}$-q.l.s.c.), are defined (see Ref. \cite{Tr5}).
We observe that Proposition 2(i) in Ref. \cite{Tr5} can be generalized as follows: 
\begin{prop}\label{a}
Let $u\in {\cal{H}}$ with s.${\cal{E}}$-q.l.s.c. $m$-version $\widetilde{u}$ and suppose further that $e_u$ exist. 
Then for any $\varepsilon>0$
\begin{eqnarray*}
\mbox{Cap}_{1,\widehat G_1 \varphi}(\{\widetilde{u} >\varepsilon\})\le \varepsilon^{-1}
\int_E e_u \,\varphi\,dm.
\end{eqnarray*}
\end{prop}
\begin{proof} We have 
\begin{eqnarray*}
\mbox{Cap}_{1,\widehat G_1 \varphi}(\{\widetilde{u} >\varepsilon\}) &\le& \lim_{k\to \infty}\lim_{\alpha\to \infty}{\cal{E}}_1(1\wedge G_1(k\varphi), (\widehat G_1 \varphi)^{\alpha}_{\{\widetilde{u} >\varepsilon\}}).
\end{eqnarray*}
Since $1\wedge G_1(k\varphi)\le 1\le \varepsilon^{-1} e_u$ $m$-a.e. on $\{\widetilde{u} >\varepsilon\}$ for any $k\in \mathbb{N}$, and 
$(\widehat G_1 \varphi)^{\alpha}_{\{\widetilde{u} >\varepsilon\}}$ is $1$-coexcessive, and since $e_u$ is $1$-excessive and 
 $(\widehat G_1 \varphi)^{\alpha}_{\{\widetilde{u} >\varepsilon\}}\le \widehat G_1 \varphi$ $m$-a.e. for any $\alpha>0$ the result 
 now easily follows.
\end{proof}\\
{\bf From now on} we fix a {\it generalized Dirichlet form $({\cal{E}},{\cal{F}})$ that is strictly quasi-regular} 
(see Definition 2 in Ref. \cite{Tr5}). 
Using Proposition \ref{a} it is not difficult to see that strict versions of statements in Ref. \cite{St1} hold as stated in the following Lemma \ref{l1}. However, we remark that the strict quasi-regularity in Lemma \ref{l1} is only used to ensure the existence of a strict ${\cal{E}}$-nest of compact metrizable sets for the proof of (ii) (cf. III. Proposition 3.2 in Ref. \cite{St1}).
\begin{lem}\label{l1}
\begin{itemize}
	\item[(i)] Let $S$ be a countable family of s.${\cal{E}}$-q.c. functions (resp. s.${\cal{E}}$-q.l.s.c. functions). Then there exists 
	a s.${\cal{E}}$-nest $(F_k)_{k\geq 1}$ such that $S\subset C_{\infty}(\{F_k\})$ (resp. $S\subset C_{l,\infty}(\{F_k\})$ ). 
	\item[(ii)] If $f$ is s.${\cal{E}}$-q.s.l.c. and $f\leq 0$ $m$-a.e. on an open set $U\subset E$, then $f\leq 0$ s.${\cal{E}}$-q.e. on $U.$ 
	If $f,g$ are s.${\cal{E}}$-q.c. and $f=g$ $m$-a.e. on an open set $U\subset E$, then $f=g$ s.${\cal{E}}$-q.e. on $U.$
\item[(iii)]  Let $u_n\in {\cal{H}}$ with s.${\cal{E}}$-q.c. $m$-version $\widetilde{u}_n$, $n\ge 1$, such that $e_{u_n-u}+e_{u-u_n}\to 0$ in ${\cal{H}}$ as $n\to \infty$ 
for some $u\in{\cal{H}}$. Then there is a subsequence $(\widetilde{u}_{n_k})_{k\ge 1}$ and a s.${\cal{E}}$-q.c. $m$-version $\widetilde{u}$ of $u$ such that $\lim_{k\ge 1}\widetilde{u}_{n_k}=\widetilde{u}$ 
s.${\cal{E}}$-quasi-uniformly. 
\item[(iv)] Let $u_n\in {\cal{F}}$ with s.${\cal{E}}$-q.c. $m$-version $\widetilde{u}_n$, $n\ge 1$, and $u_n\to u$ in ${\cal{F}}$.  Then 
there is a subsequence $(\widetilde{u}_{n_k})_{k\ge 1}$ and a s.${\cal{E}}$-q.c. $m$-version $\widetilde{u}$ of $u$ such that $\lim_{k\ge 1}\widetilde{u}_{n_k}=\widetilde{u}$ 
s.${\cal{E}}$-quasi-uniformly. 
\end{itemize}
\end{lem}
By strict quasi-regularity one can find a strict ${\cal{E}}$-nest of compact metrizable sets $(E_k)_{k\ge 1}$ as in IV. Lemma 1.10 in Ref. \cite{St1} (see also Ref. \cite{mr}). Let
$$
Y_1:=\bigcup_{k\in \mathbb{N}} E_k.
$$ 
Then $Y_1$ is a Lusin space. Since 
$E\setminus Y_1$ is strictly ${\cal{E}}$-exceptional it is ${\cal{E}}$-exceptional, hence $m(E\setminus Y_1)=0$ and 
we may identify $L^2(E;m)$ with $L^2(Y_1,m)$ canonically. \\
By Lemma 2 in Ref. \cite{Tr5} we know that for any $\alpha>0$ there exists a kernel $\widetilde{R}_{\alpha}$ 
from $(E,{\cal{B}}(E))$ to $(Y_1,{\cal{B}}(Y_1))$ such that 
\begin{enumerate}
\item[(R1)] $\alpha\widetilde{R}_{\alpha}(z,Y_1)\le 1$ for all $z\in E$. 
\item[(R2)] $\widetilde{R}_{\alpha}f$ is a s.${\cal{E}}$-q.c. $m$-version of $G_{\alpha}f$ for all (measurable) $f\in {\cal{H}}$. 
\end{enumerate}
Moreover, the kernel $\widetilde{R}_{\alpha}$ is unique in the sense that, if $K$ is another kernel 
from $(E,{\cal{B}}(E))$ to $(Y_1,{\cal{B}}(Y_1))$ satisfying $(R1)$ and $(R2)$, it follows that 
$K(z,\cdot)=\widetilde{R}_{\alpha}(z,\cdot)$ s.${\cal{E}}$-q.e. \\ 
The next Proposition \ref{sp}(ii) is even when applied to the sectorial case an improvement over Lemma 3.4 in Ref. \cite{mrz} since we can choose 
the function $h$ of Lemma 3.4 in Ref. \cite{mrz} to be in the domain of the infinitesimal generator. Note also that Proposition \ref{sp}(ii) is a statement about existence and  not stated for any $\varphi$ with the given properties.
\begin{prop}\label{sp}
\begin{itemize}
	\item[(i)] Let $(u_n)_{n\ge 1}\subset {\cal{H}}$ densely. Then $\{\widetilde{R}_{1}u_n^+, \widetilde{R}_{1}u_n^-; n\ge 1\}$ 
	separates the points of $E_{\Delta}\setminus N$, where $N$ is some s.${\cal{E}}$-exceptional set.
	\item[(ii)] There is some $\varphi\in L^1(E,m)\cap {\cal{B}}$ such that $0<\varphi(z)\le 1$ for every $z\in E$, and such that $\widetilde{R}_{1}\varphi >0$ s.${\cal{E}}$-q.e.
\end{itemize}
\end{prop}
\begin{proof}
(i) Using Lemma \ref{l1}(iv) the proof is the same as in IV. Proposition 1.9 in Ref. \cite{St1}.\\
(ii) Choose $(u_n)_{n\ge 1}\subset L^1(E,m)\cap {\cal{B}}_b$ such that $(u_n)_{n\ge 1}\subset {\cal{H}}$ densely. Then by (i) 
$\{\widetilde{R}_{1}u_n^+, \widetilde{R}_{1}u_n^-; n\ge 1\}$ separates the points of $E_{\Delta}\setminus N$, where $N$ is some s.${\cal{E}}$-exceptional set.
Define
$$
h(x):=\sum_{n\ge 1} c_n\widetilde{R}_{1}(u_n^+ + u_n^-)(x), \ \ \mbox{ with } \ c_n :=2^{-n}(1+\|u_n \|_{L^{1}(E,m)}+\|u_n \|_{L^{\infty}(E,m)}). 
$$
Since $\{\widetilde{R}_{1}u_n^+, \widetilde{R}_{1}u_n^-; n\ge 1\}$ separates the points of $E_{\Delta}\setminus N$ we have
 $h(x)>0$ for all $x\in E\setminus N$. Since $g_k:=\sum_{n = 1}^k c_n (u_n^+ + u_n^-)$ converges in $L^{1}(E,m)$ to some $g$ with $0\le g \le 1$ and $\widetilde{R}_{1}$ is 
 a kernel we obtain $h=\widetilde{R}_{1}g$. Now choose $\rho\in L^{1}(E,m)$ with $0< \rho\le 1$. Then $\varphi:=\rho \vee g$ is the desired function.
\end{proof}
{\bf From now on} we assume that the strictly quasi-regular generalized Dirichlet form ${\cal{E}}$ satisfies {\bf D3}. 
Using Lemma \ref{l1}, Proposition \ref{sp}, and the  strict version of IV. Proposition 2.8 in Ref. \cite{St1} we obtain the following:
\begin{lem}\label{J0}
There exists a countable family $J_0$ of bounded strictly ${\cal{E}}$-quasi-continuous 1-excessive functions and a Borel set $Y\subset Y_1$ satisfying:
\begin{enumerate}
\item[(i)] If $u,v\in J_0,\ \alpha,c_1,c_2\in\mathbb{Q}_+^*$, then $\widetilde{R}_{\alpha}u,\ u\wedge v,\ u\wedge 1,\ (u+1)\wedge v,\ c_1 u+c_2 v$ are all in $J_0$.
\item[(ii)] $N:=E\setminus Y$ is strictly ${\cal{E}}$-exceptional and $\widetilde{R}_{\alpha}(x,N)=0,$ for all $x\in Y,\ \alpha\in\mathbb{Q}_+^*$.
\item[(iii)] $J_0$ separates the points of $Y_{\Delta}$.
\item[(iv)] If $u\in J_0,\ x\in Y$, then $\beta \widetilde{R}_{\beta+1}u(x)\leq u(x)$ for all $\beta\in\mathbb{Q}_+^*$,\\
$\widetilde{R}_{\alpha}u(x)-\widetilde{R}_{\beta}u(x)=(\beta-\alpha)\widetilde{R}_{\alpha}\widetilde{R}_{\beta}u(x)$ for all $\alpha,\beta\in\mathbb{Q}_+^*$,\\
$\lim_{\mathbb{Q}_+^*\ni \alpha\rightarrow\infty}\alpha \widetilde{R}_{\alpha+1}u(x)=u(x).$\\ 
\end{enumerate}
\end{lem}
Next, we define for $\alpha\in\mathbb{Q}_+^*,\ A\in{\cal{B}}(Y_{\Delta}):={\cal{B}}(E_{\Delta})\cap Y_{\Delta}$
\begin{equation}\label{k1}
R_{\alpha}(x,A) :=  
\begin{cases}
\widetilde{R}_{\alpha}(x,A\cap Y)+\left ( \frac{1}{\alpha} -\widetilde{R}_{\alpha}(x,Y)\right )1_A(\Delta),& if\ x\in Y \\
\frac{1}{\alpha} 1_A(\Delta),& if x=\Delta 
\end{cases} 
\end{equation}
and set 
\begin{equation}\label{k2}
J:=\{u+c1_{Y_{\Delta}}\mid u\in J_0, c\in\mathbb{Q}_+\}.
\end{equation}
Since $J_0$ separates the points of $Y_{\Delta}$, so does $J$. The following lemma is also clear. 
\begin{lem}\label{J}
 Let $(R_{\alpha})_{\alpha\in\mathbb{Q}_+^*}$ and J be as in (\ref{k1}), (\ref{k2}). Then the statements of Lemma \ref{J0} remain true with 
 $J_0,\ Y$ and $\widetilde{R}_{\alpha}$ replaced by $J,\ Y_{\Delta}$ and $R_{\alpha}$ respectively.
\end{lem}
\subsection{The construction of nice excessive functions}\label{2.3}
Since strict quasi-regularity implies quasi-regularity by Proposition 2(ii) in Ref. \cite{Tr5}, and since {\bf D3} is in force, we obtain 
by IV.Theorem 2.2 in Ref. \cite{St1} 
that $({\cal{E}},{\cal{F}})$ is associated with some $m$-tight $m$-special standard process. We denote the process resolvent by
$$
V_{\alpha} f(z)=E_z\left [\int_0^{\infty}e^{-\alpha t}f(Y_t)dt\right ], \ \ \ \ \alpha>0,\ f\in{\cal{H}} \cap L^{\infty}(E,m).
$$
By Remark \ref{all}(ii) the strict capacity does not depend on the special choice of $\varphi$. We may and will hence {\bf from now on} assume that $\varphi$ is as in Proposition \ref{sp}(ii). 
The following two lemmas are crucial for the later study of weak limits.
\begin{lem}\label{gf}
Let $U_n\subset E,\ n\geq 1$ be a decreasing sequence of open sets such that $\mbox{Cap}_{1,\widehat G_1 \varphi}(U_n)\rightarrow 0$, as $n\rightarrow\infty$. 
Then we can find $m$-versions $e_n$ of $e_{U_n}$ such that:
\begin{enumerate}
\item[(i)] $e_n\geq 1$ ${\cal{E}}$-q.e. on $U_n,\ n\geq 1.$ In particular, there are ${\cal{E}}$-exceptional sets $N_n\in {\cal{B}}(E)$, $N_n\subset U_n$, such that 
\begin{equation*}
\widehat{e}_n(x):=e_n(x)+ 1_{N_n}(x)\ge 1 
\ \ \forall x\in U_n, \ n\ge 1.
\end{equation*}
\item[(ii)] $\alpha \widetilde{R}_{\alpha+1}e_n\leq e_n$ and $\alpha \widetilde{R}_{\alpha+1}\widehat{e}_n\leq \widehat{e}_n$ s.${\cal{E}}$-q.e. for any  $\alpha\in\mathbb{Q}_+^*,\ n\geq 1.$
\item[(iii)] $e_n\searrow 0$ and $\widehat{e}_n\to 0$ s.${\cal{E}}$-q.e. as $n\rightarrow \infty$.
\end{enumerate}
\end{lem}
\begin{proof}
Define for $n\ge 1$
\begin{equation}\label{en}
e_n:=\sup_{\alpha\ge 1}\sup_{l\ge 1}\alpha \widetilde{R}_{\alpha +1}\big(\left (1\wedge V_1(l\varphi)\right )_{U_n}\big).
\end{equation}
where $\left (1\wedge V_1(l\varphi)\right )_{U_n}$ is some bounded measurable $m$-version of $\left (1\wedge G_1(l\varphi)\right )_{U_n}$. 
Clearly $e_n$ is an $m$-version of $e_{U_n}$. 
Since 
$\left (1\wedge G_1(l\varphi)\right )_{U_n}$ is $1$-excessive, and $\widetilde{R}_{\alpha+1}f$ is s.${\cal{E}}$-q.c. 
for any (measurable) $f\in {\cal{H}}$ by (R2),  
by Lemma \ref{l1}(ii) it is clear that the first part of (ii) holds. The second part of (ii) similarly also holds once we have shown that 
$N_n$ is ${\cal{E}}$-exceptional, hence in particular $m$-negligible. This is done at the end of the proof. \\
Obviously  $e_n$ is s.${\cal{E}}$-q.l.s.c,  s.${\cal{E}}$-q.e. decreasing in $n$, $\lim_{n\to \infty} e_n$ exists s.${\cal{E}}$-q.e. and $\lim_{n\to \infty} e_n\ge 0$ s.${\cal{E}}$-q.e.  
We have 
\begin{eqnarray*}
\mbox{Cap}_{1,\widehat G_1 \varphi}\left (\{\lim_{n\to \infty} e_n>0\}\right ) 
& \le  & \sum_{k\ge 1}\mbox{Cap}_{1,\widehat G_1 \varphi}\left (\cap_{n\ge 1}\left \{e_n>k^{-1}\right \}\right ). \\
\end{eqnarray*}
Up to some ${\cal{E}}$-exceptional set $\left \{e_n>k^{-1}\right \}\subset  \bigcup_{\alpha\ge 1}\bigcup_{l\ge 1}\left \{\alpha \widetilde{R}_{\alpha +1}\big (\left (1\wedge V_1(l\varphi)\right )_{U_n}\big )>k^{-1}\right \}$,
and $\left \{\alpha \widetilde{R}_{\alpha +1}\big (\left (1\wedge V_1(l\varphi)\right )_{U_n}\big )>k^{-1}\right \}$ is increasing in $\alpha$, $l$. Then, since $\mbox{Cap}_{1,\widehat G_1 \varphi}$ is a Choquet capacity we obtain with Proposition \ref{a} 
\begin{eqnarray*}
\mbox{Cap}_{1,\widehat G_1 \varphi}\left (\left \{e_n>k^{-1}\right \}\right )
&\le& k\cdot \sup_{\alpha\ge 1}\sup_{l\ge 1} \int_E \alpha \widetilde{R}_{\alpha +1}\big (\left (1\wedge V_1(l\varphi)\right )_{U_n}\big )\varphi \,dm. \\
\end{eqnarray*}
Therefore
\begin{eqnarray*}
\mbox{Cap}_{1,\widehat G_1 \varphi}\left (\cap_{n\ge 1}\left \{e_n>k^{-1}\right \}\right )
& \le  & k\cdot\inf_{n\ge 1} \int_E e_n \,\varphi \,dm=0. \\
\end{eqnarray*}
Thus the first part of (iii) holds. The second part of (iii) is clear since $\limsup_{n\ge 1}1_{N_n}\le 1_{\cap_{n\ge 1}U_n}=0$ s.${\cal{E}}$-q.e.\\
Since $\widetilde{R}_{\alpha+1}f$, $V_{\alpha+1}f$, $f\in {\cal{H}}\cap L^{\infty}(E,m)$, are ${\cal{E}}$-q.c. and coincide $m$-a.e. by (R2), it 
follows by III.Corollary 3.4 in  Ref. \cite{St1}  that $V_{\alpha+1}f=\widetilde{R}_{\alpha+1}f$ ${\cal{E}}$-q.e. Using 
$\left (1\wedge V_1(l\varphi)\right )_{U_n}=1\wedge V_1(l\varphi)$ $m$-a.e. on $U_n$, and $1\wedge V_1(l\varphi)\nearrow \mathbb{I}_E$ as
 $l\to\infty$, it follows ${\cal{E}}$-q.e.
\begin{equation}\label{e2}
e_n \ge \limsup_{\alpha\ge 1}\alpha V_{\alpha +1} \mathbb{I}_{U_n}.
\end{equation}
By right-continuity and normality of the process $Y$ we obtain for all $z\in U_n$
$$
\lim_{\alpha\to\infty}\alpha V_{\alpha +1} \mathbb{I}_{U_n}(z)=\lim_{\alpha\to\infty}\frac{\alpha}{\alpha+1}E_z\left [\int_0^{\infty}e^{-t}\mathbb{I}_{U_n}(Y_{\frac{t}{\alpha+1}})dt\right ]=1.
$$ 
Hence the first part of (i) holds. For the second part of (i) we can find ${\cal{E}}$-exceptional sets $N_n\in {\cal{B}}(E)$, $N_n\subset U_n$, with $e_n \cdot 1_{U_n\setminus N_n}+1_{N_n}\ge 1$ pointwise on $U_n$. 
But  $e_n\ge 0$ everywhere since $\widetilde{R}_{\alpha +1}$ is a kernel and so we obtain $\widehat{e}_n\ge 1$ on $U_n$ as desired.
\end{proof}
\begin{rem}\label{enr}
In Lemma \ref{gf}(i) we were not able to show directly 
\begin{equation}\label{enh}
e_n\ge 1 \mbox{ s.}{\cal{E}}\mbox{-q.e. on }\ U_n, \ n\ge 1.
\end{equation}
(Unfortunately, $\left (1\wedge V_1(l\varphi)\right )_{U_n}$ has only a s.${\cal{E}}$-q.l.s.c. $m$-version in general and the inequality in Lemma \ref{l1}(ii) is just the wrong way around.) (\ref{enh}) is used in Lemma 3.5, Lemma 3.6, and proof of Theorem 3.3 in Ref. \cite{mrz} in an essential way.  Instead, we will use the functions $\widehat{e}_n$ defined  in Lemma \ref{gf}(i) which is sufficient 
(cf. Lemmas \ref{ss} and \ref{liq}, and Theorem \ref{key}). We remark that it is even sufficient to only know that $e_n\ge 1$ $m$-a.e. on $U_n$, so that the sets $N_n$ in Lemma \ref{gf}(i) are only $m$-negligible.\\
It will turn out a posteriori that (\ref{enh}) actually holds. In fact, by our main result Theorem \ref{hunt2} below it follows that the process resolvent 
$V_{\alpha+1}f$, $f\in {\cal{H}}\cap L^{\infty}(E,m)$, is s.${\cal{E}}$-q.c. Thus applying Lemma \ref{l1}(ii) $V_{\alpha+1}f=\widetilde{R}_{\alpha+1}f$ s.${\cal{E}}$-q.e.
Therefore (\ref{e2}) holds s.${\cal{E}}$-q.e. and (\ref{enh}) follows. 
\end{rem}
\begin{lem}\label{ss}
In the situation of Lemma \ref{gf} there exists $S\in{\cal{B}}(E),\ S\subset Y$ such that $E\setminus S$ is strictly ${\cal{E}}$-exceptional and the following holds:
\begin{enumerate}
\item[(i)] $\widetilde{R}_{\alpha}(x,Y\setminus S)=0\ \forall x\in S,\ \alpha\in\mathbb{Q}_+^*.$
\item[(ii)] $\widehat{e}_n(x)\geq 1\ for\ x\in U_n,\ n\geq 1$, and $\widetilde{R}_{\alpha} 1_{N_n}(x)=0$ $\forall x\in S,\ \alpha\in\mathbb{Q}_+^*, n\geq 1$.
\item[(iii)]$\alpha\widetilde{R}_{\alpha +1}\widehat{e}_n(x)\leq \widehat{e}_n(x) \ \forall x\in S,\ \alpha\in\mathbb{Q}_+^*,\ n\geq 1$.
\item[(iv)] $\widehat{e}_n(x)\to 0$ as $n\to \infty\ \ \forall x\in S$.
\end{enumerate} 
\end{lem}
\begin{proof}
The first assertion of (ii) holds by definition in Lemma \ref{gf}(i). 
By Lemma \ref{gf}, $(R2)$, and Lemma \ref{l1}(ii), the rest of the proof works as in IV.3.11 of Ref. \cite{mr}.
\end{proof} 

\section{The approximating forms ${\cal{E}}^{\beta}$ and the approximating processes $X^{\beta}$}\label{3}
Let $J,\ Y_{\Delta}$ and $(R_{\alpha})_{\alpha\in\mathbb{Q}_+^*}$ be as in Lemma \ref{J}. First, we collect some results 
of Chapter 4 section 2 of Ref. \cite{ek}. For a fixed $\beta\in\mathbb{Q}_+^*$, let $\{Y^{\beta}(k),\ k=0,1,\ldots\}$ be a Markov chain 
in $Y_{\Delta}$ with initial distribution 
$\nu$ and transition function $\beta R_{\beta}$. Let further $(\Pi^{\beta}_t)_{t\geq 0}$ be a Poisson process with parameter $\beta$ 
and independent of $\{Y^{\beta}(k),\ k=0,1,\ldots\}$. Then it is known that
$$
X^{\beta}_t:=Y^{\beta}(\Pi^{\beta}_t)
$$
is a strong Markov process in $Y_{\Delta}$ with transition semigroup
\begin{equation}\label{ak} 
P^{\beta}_t f=e^{-\beta t}\sum_{k=0}^{\infty}\frac{(\beta t)^k}{k!}(\beta R_{\beta})^k f\quad \forall t\geq 0,\ f\in{\cal{B}}_b(Y_{\Delta}),
\end{equation}
i.e. we have for all $t,s\ge 0, f\in {\cal{B}}_b(Y_{\Delta})$
\begin{equation}\label{ts}
E[f(X^{\beta}_{t+s})\mid \sigma(X^{\beta}_u,u\le t)]=(P^{\beta}_s f)(X^{\beta}_t).
\end{equation}
Here (\ref{ts}) easily follows from (2.14) of Chapter 4 in Ref. \cite{ek}.
Furthermore from the formula (\ref{ak}) one can see that $(P_t^{\beta})_{t\geq 0}$ is a strongly continuous contraction semigroup on the Banach space $({\cal{B}}_b(Y_{\Delta}),\|\cdot\|_{\infty})$. 
The corresponding generator is 
\begin{equation}\label{g}
L^{\beta}f(x)=\frac{d}{dt}P^{\beta}_tf(x)\big|_{t=0} =\beta(\beta R_{\beta}f(x)-f(x)),\quad f\in{\cal{B}}_b(Y_{\Delta}).
\end{equation}
Define the forms ${\cal{E}}^{(\beta)},\ \beta >0,$ by 
$$
{\cal{E}}^{(\beta)}(u,v):=\beta (u-\beta G_{\beta}u,v)_{{\cal{H}}},\quad u,v\in{\cal{H}},
$$ 
where we recall that $(G_{\beta})_{\beta >0}$ is the $L^2$-resolvent of ${\cal{E}}$.
It is known (see e.g. Chapter I in Ref. \cite{mr}), that the $C_0$-semigroup of submarkovian contractions on $L^2(E;m)$ that is associated to ${\cal{E}}^{(\beta)}$ is given by 
\begin{equation}\label{as} 
T_t^{\beta}f=e^{-\beta t}\sum_{j=0}^{\infty}\frac{(\beta t)^j}{j!}(\beta G_{\beta})^j f, \quad f\in {\cal{H}}.
\end{equation}
From (\ref{ak}), (\ref{ts}), and (\ref{as}) it follows that $(X^{\beta}_t)$ is associated with ${\cal{E}}^{(\beta)}$. 
Since $R_{\beta}f$ is an $m$-version of $G_{\beta}f$ for any measurable $f\in {\cal{H}}$, by I. Examples 4.9(ii) in Ref. \cite{St1} we see that ${\cal{E}}^{(\beta)}$ is a generalized Dirichlet form and 
\begin{eqnarray*}
{\cal{E}}^{(\beta)}(u,v)=(-L^{\beta}u,v)_{{\cal{H}}},\quad u,v\in{\cal{H}}.\\
\end{eqnarray*}
For an arbitrary subset $M\subset E_{\Delta}$ let $\Omega_{M}:=D_{M}[0,\infty)$ 
be the space of all c\`adl\`ag functions from $[0,\infty)$ to $M$. 
Let $(X_t)_{t\geq 0}$ be the coordinate process on $\Omega_{E_{\Delta}}$, i.e. $X_t(\omega)=\omega(t)$ for $\omega\in\Omega_{E_{\Delta}}$. 
$\Omega_{E_{\Delta}}$ is equipped with the Skorokhod topology (see Chapter 3 in Ref. \cite{ek}).
Let 
$P_x^{\beta}$ be the law of $X^{\beta}$ on $\Omega_{E_{\Delta}}$ with initial distribution $\delta_x$ if $x\in Y_{\Delta}$, 
and if $x\in E_{\Delta}\setminus Y_{\Delta}$ let $P^{\beta}_x$ be the Dirac measure on 
$\Omega_{E_{\Delta}}$ such that $P_x^{\beta}[X_t=x\mbox{ for all }t\geq 0]=1$. Finally, let $({\cal{F}}^{\beta}_t)_{t\geq 0}$ be the completion w.r.t. 
$(P_x^{\beta})_{x\in E_{\Delta}}$ of the natural filtration of $(X_t)_{t\geq 0}$.
\begin{prop}\label{ah}
$\mathbb{M}^{\beta}:=(\Omega_{E_{\Delta}},\ (X_t)_{t\geq 0},\ ({\cal{F}}^{\beta}_t)_{t\geq 0},\ (P^{\beta}_x)_{x\in E_{\Delta}})$ is a Hunt process associated with ${\cal{E}}^{(\beta)}$, i.e. for all $t\geq 0$ and any m-version of $u\in L^2(E;m)$, $x\mapsto \int u(X_t)\; dP^{\beta}_x$ is an m-version of $T^{\beta}_t u.$
\end{prop}
\begin{proof} By construction $\mathbb{M}^{\beta}$ is a right process that has left limits in $E_{\Delta}$. The quasi-left continuity up to $\infty$ 
can be shown by a routine argument following Ref. \cite{get} (cf. IV.3.21 in Ref. \cite{mr}).
\end{proof}
\centerline{}
Let $J=\{u_n\mid n\in\mathbb{N} \}$ and 
$$
g_n:=R_1 u_n,\ n\in \mathbb{N}.
$$
Define for all $x,y\in Y_{\Delta}$ 
$$
\rho(x,y)=\sum_{n=1}^{\infty}\frac{1}{2^n}\lvert g_n(x)-g_n(y)\rvert\wedge 1.
$$
By Lemma \ref{J0}(ii) and Lemma \ref{J} $\{g_n\mid n\in\mathbb{N}\}$ separates the points of $Y_{\Delta}$ and 
hence $\rho$ defines a metric on $Y_{\Delta}$. 
We may assume that $Y_{\Delta}$ is a Lusin topological space (cf. IV.Remark 3.2(iii) in Ref. \cite{mr}). It follows by Lemma 18 on p.108 of Ref. \cite{schwartz} 
that ${\cal{B}}(Y_{\Delta})=\sigma(g_n\mid n\in\mathbb{N})=(\rho\mbox{--}){\cal{B}}(Y_{\Delta})$. 
Now define
$$
\overline{E}:=\overline{Y}_{\Delta}^{\rho}.
$$
$(\overline{E},\rho)$ is a compact metric space by Tychonoff's theorem.\\
We extend the kernel $(R_{\alpha})_{\alpha\in\mathbb{Q}_+^*}$ to the space $\overline{E}$ by setting for $\alpha \in\mathbb{Q}_+^*,\ A\in{\cal{B}}(\overline{E})$,
\begin{equation}\label{k1bis}
R_{\alpha}(x,A) :=  \begin{cases}R_{\alpha}(x,A\cap Y_{\Delta}),& x\in Y_{\Delta}\\ \frac{1}{\alpha}1_A(x),& x\in \overline{E}\setminus Y_{\Delta}.\end{cases}\\ 
\end{equation}
We may regard $(X^{\beta}_t)_{t\geq 0}$ as a c\`adl\`ag process with state space $\overline{E}$ and use the same notation as before: 
$P_x^{\beta}$ denotes hence the law of $(X^{\beta}_t)_{t\geq 0}$ in $\Omega_{\overline{E}}$ with initial distribution $\delta_x$. 
Each $g_n$ is $\rho$-uniformly continuous and extends therefore uniquely to a continuous function on $\overline{E}$ which we denote again by $g_n$.\\
For the convenience of the reader we include the proof of the following theorem, which as we feel is slightly more transparent than the corresponding proof of Theorem 3.2 in Ref. \cite{mrz}.
\begin{theo}\label{tt}
$\{P_x^{\beta}\mid \beta\in\mathbb{Q}_+^*\}$ is relatively compact for any $x\in \overline{E}$.
\end{theo}
\begin{proof}
We first show that assumptions of 9.4 Theorem of Chapter 3 in Ref. \cite{ek} are fulfilled with $C_a=C(\overline{E})$ (where $C_a$ is 
as in the just mentioned theorem of Ref. \cite{ek}). 
Since $g_n\in D(L^{\beta})$ it follows that 
$$
\left (g_n(X^{\beta}_t)-\int_0^t L^{\beta}g_n(X^{\beta}_s)\; ds\right )_{t\geq 0}
$$ 
is an $(P_x^{\beta}, (\mathcal{F}_t^{\beta})_{t\ge 0})$-martingale for any $x\in \overline{E}$. 
Since $L^{\beta}g_n=1_{Y_{\Delta}}\beta R_{\beta}(g_n-u_n)$ 
we have for all $n\in\mathbb{N}$
\begin{eqnarray*}
\sup_{\beta\in\mathbb{Q}_+^*} \|L^{\beta} g_n\|_{\infty} = \sup_{\beta\in\mathbb{Q}_+^*}\|1_{Y_{\Delta}}\beta R_{\beta}(g_n-u_n)\|_{\infty}\leq \|1_{Y_{\Delta}}(g_n-u_n)\|_{\infty} < +\infty.
\end{eqnarray*}
So, we proved that $R_1 J:=\{g_n\mid n\in\mathbb{N}\}\subset D$ where $D\subset C(\overline{E})$ is the linear space from the theorem in Ref. \cite{ek}. Since for any $u\in J$
$$
R_1 J \ni R_1\left (\alpha(u-\alpha R_{\alpha +1}u\right ))(x)=\alpha R_{\alpha +1}u(x)\nearrow u(x),\quad \mathbb{Q}_+^*\ni \alpha\rightarrow\infty, \forall x\in Y_{\Delta},
$$
we see by Dini's theorem that every $u\in J$ has a unique ($\rho$-uniformly) continuous extension to $\overline{E}$ that is again denoted by $u$. 
Thus we may and do consider $J$  as a subset of $C(\overline{E})$. In particular, 
if $\overline{A}^{\|\cdot\|_{\infty}}$ denotes the uniform closure of $A\subset C(\overline{E})$, we have 
$$
\overline{J-J}^{\|\cdot\|_{\infty}}\subset \overline{R_1(J-J)}^{\|\cdot\|_{\infty}}\subset \overline{D}^{\|\cdot\|_{\infty}}\subset C(\overline{E}).
$$
Since $J-J$ contains the constant functions, is inf-stable and separates the points of $\overline{E}$ we obtain that $J-J$ is dense in $C(\overline{E})$ by the Stone-Weierstra\ss\ theorem.
Hence $\overline{D}^{\|\cdot\|_{\infty}}= C(\overline{E})$ and so by the theorem of Ref. \cite{ek} $\{f\circ X^{\beta}\mid \beta\in \mathbb{Q}_+^*\}$ is relatively compact for all $f\in C(\overline{E})$. 
Since $\overline{E}$ is compact, the compact containment condition trivially holds and so by 9.1 Theorem of Chapter 3 in Ref. \cite{ek} $\{X^{\beta}\mid \beta\in \mathbb{Q}_+^*\}$ is relatively compact as desired.
\end{proof}
\section{Limiting process associated with the strictly quasi-regular generalized Dirichlet form}\label{4}
For a Borel subset $S\subset Y$, we write $S_{\Delta}:=S\cup\{\Delta\}$. On $S_{\Delta}$ we consider (except otherwise stated) the topology induced by the metric $\rho$. In particular the $\rho$-topology 
and the original one generate the same Borel $\sigma$-algebra on $S_{\Delta}$.\\
Let $\mathbb{M}^{\beta}:=(\Omega_{E_{\Delta}},\ (X_t)_{t\geq 0},\ ({\cal{F}}_t^{\beta})_{t\geq 0},\ (P^{\beta}_x)_{x\in E_{\Delta}})$ 
be the Hunt process from Proposition \ref{ah}, $R_{\alpha}$ and $\overline{E}$ be as in Section \ref{3}. Let $S\in{\cal{B}}(E)$ and $\widehat{e}_n$ be as in Lemma \ref{ss}.
In view of Lemma \ref{ss}(i), exactly as in Lemma 3.7 of Ref. \cite{mrz} one can show that 
\begin{equation}\label{iv}
P_x^{\beta}[X_t\in S_{\Delta},\ X_{t-}\in S_{\Delta}\ \forall t\geq 0]=1\ \ \forall x\in S_{\Delta}.
\end{equation}
Due to Lemma \ref{ss}(i), (iii) the proof of the next lemma is also the same as the proof of Lemma 3.8 in Ref. \cite{mrz}.
\begin{lem}\label{liq}
Let $\beta\in\mathbb{Q}_+^*,\ \beta \geq 2,\ n\geq 1.$ Then $\widehat{e}_n$ is a $(P^{\beta}_t)$-2-excessive function on $S_{\Delta}$, i.e. 
$e^{-2t}P_t^{\beta}\widehat{e}_n(x)\leq \widehat{e}_n(x)$ and $\lim_{t\rightarrow 0}e^{-2t}P_t^{\beta}\widehat{e}_n(x)=\widehat{e}_n(x)\ \ \forall x\in S_{\Delta}$.
\end{lem}
\begin{rem}\label{mis}
The proof of Lemma 3.7 in Ref. \cite{mrz} (resp. Lemma 3.2 in Ref. \cite{mrs}) contains an inaccuracy, namely it is not as stated true that $X_{\tau_n}\in U_n$ $P^{\beta}_x$-a.s. on $\{\tau_n < \infty\}$, where 
$$
\tau_n:=\inf\{t\geq 0\mid X_t\in U_n\},\ \ n\in \mathbb{N}.
$$
By right-continuity $X_{\tau_n}$ will be in general in the closure $\overline{U}_n$ of $U_n$. This leads to a wrong argument 
so that the proof of Lemma 3.7 in Ref. \cite{mrz} (resp. Lemma 3.2 in Ref. \cite{mrs}) cannot be maintained. 
However, these lemmas are quite important since they tell us that any s.${\cal{E}}$-nest is a pointwise strict ${\cal{E}}^{\beta}$-nest. 
Following the proof of Lemma \ref{liq2} below one can see how this inaccuracy can be corrected. Therefore the statement of 
Lemma 3.7 in Ref. \cite{mrz} (resp. Lemma 3.2 in Ref. \cite{mrs}) remains true and no results of Ref. \cite{mrz} (resp. Ref. \cite{mrs}) are affected. 
\end{rem} 

\begin{lem}\label{liq2}
Let $\beta\in \mathbb{Q}_+^*,\ \beta\geq 2$. 
Then $E^{\beta}_x[e^{-2\tau_n}]\leq \widehat{e}_n(x)\ \ \forall x\in S_{\Delta}$.
\end{lem}
\begin{proof}
Since by (\ref{iv}) $S_{\Delta}$ is invariant set of $\mathbb{M}^{\beta}$, the restriction $\mathbb{M}^{\beta}_{S_{\Delta}}$ of $\mathbb{M}^{\beta}$ 
to $S_{\Delta}$ is still a Hunt process. 
Since $\widehat{e}_n$ is a $(P^{\beta}_t)$-2-excessive function on $S_{\Delta}$ we have that $(e^{-2t}\widehat{e}_n(X_t))_{t\geq 0}$ 
is a positive right-continuous $(P_x^{\beta}, ({\cal{F}}_t^{\beta})_{t\ge 0})$-supermartingale for all $ x\in S_{\Delta}$ 
(use IV. Proposition 5.14 claim 1 of Ref. \cite{mr} and monotone classes as indicated in the proof of 5.8(iii) in Ref. \cite{get}). 
By the optional sampling theorem and normality we have
$$
E^{\beta}_x[e^{-2\tau_n}\widehat{e}_n(X_{\tau_n})]\leq \widehat{e}_n(x),\ x\in S_{\Delta}.
$$
By Lemma \ref{ss}(ii) we have that $\widehat{e}_n(x)\geq 1$ for all $x\in U_n$. 
Hence, by right-continuity for all $x\in S_{\Delta}$ we have $\widehat{e}_n(X_{\tau_n})=\lim_{\mathbb{Q}^*_+\ni t_n\downarrow \tau_n}\widehat{e}_n(X_{t_n})\geq 1\  P_x^{\beta}$-a.s. on $\{\tau_n < \infty\}$. (As usual we let $X_{\infty}:=\Delta$, and $f(\Delta):=0$ for any function $f$). It follows that for all $x\in S_{\Delta}$ 
$$
E_x^{\beta}[e^{-2\tau_n}]\leq E^{\beta}_x[e^{-2\tau_n}\widehat{e}_n(X_{\tau_n})] \leq \widehat{e}_n(x).
$$
\end{proof}\\
In view of Lemma \ref{liq2}, Lemma \ref{ss}, the proof of the following \lq\lq key theorem\rq\rq is the same as in Theorem 3.3 of Ref. \cite{mrz}.
\begin{theo}\label{key}
There exists a Borel subset $Z\subset Y$ and a Borel subset $\Omega\subset \Omega_{\overline{E}}$ with the following properties:
\begin{enumerate}
\item[(i)] $E\setminus Z$ is strictly ${\cal{E}}$-exceptional.
\item[(ii)] $R_{\alpha}(x, \overline{E}\setminus Z_{\Delta})=0\ \ \forall x\in Z_{\Delta},\ \alpha\in\mathbb{Q}_+^*.$
\item[(iii)] If $\omega\in\Omega$, then $\omega_t,\ \omega_{t-}\in Z_{\Delta}$ for all $t\geq 0$. Moreover, each $\omega\in\Omega$ is c\`adl\`ag in the 
original topology of $Y_{\Delta}$ and $\omega_{t-}^0=\omega_{t-}$ for all $t>0$, where $\omega_{t-}^0$ denotes the left limit in the original topology.
\item[(iv)] If $x\in Z_{\Delta}$ and $P_x$ is a weak limit of some sequence $(P^{\beta_j}_x)_{j\in\mathbb{N}}$ with $\beta_j \in\mathbb{Q}_+^*,\ \beta_j\uparrow\infty$, then $P_x[\Omega]=1$. 
\end{enumerate}
\end{theo}
For $\alpha, \beta\ \in\mathbb{Q}_+^*$ let 
$R_{\alpha}^{\beta}f(x):=E_x^{\beta}\left[\int_0^{\infty}e^{-\alpha t}f(X_t)\; dt\right], f\in{\cal{B}}_b(\overline{E}),\ x\in\overline{E}$.
Since the identities (\ref{ts}) and (\ref{g}) carry over to $B_b(\overline{E})$ it is straight forward to check that the resolvent of the Yosida approximation has the following form
\begin{equation*}
R_{\alpha}^{\beta}f=\left(\frac{\beta}{\alpha +\beta}\right)^2 R_{\frac{\alpha \beta}{\alpha +\beta}}f+\frac{1}{\alpha +\beta}f.
\end{equation*}
For the explicit calculations we refer to Lemma 4.1 in Ref. \cite{mrz}. It is equally straight forward to check (see Lemma 4.2 in Ref. \cite{mrz}) that if $P_x$, $x\in\overline{E}$, is 
a weak limit of a subsequence $(P_x^{\beta_j})_{j\geq 1}$ with $\beta_j\uparrow\infty,\ \beta_j\in\mathbb{Q}_+^*$, then the kernel $P_t f(x):=E_x[f(X_t)]:=\int_{\Omega} f(X_t(\omega))P_x(d\omega), f\in{\cal{B}}_b(\overline{E})$,
satisfies  
\begin{equation}\label{atd}
\int_0^{\infty}e^{-\alpha t}P_t f(x)\; dt= R_{\alpha}f(x),\quad \forall f\in{\cal{B}}_b(\overline{E}),\ \alpha\in\mathbb{Q}_+^*.
\end{equation}
In particular, the kernels $P_t,\ t\geq 0$, are independent of the subsequence $(P_x^{\beta_j})_{j\geq 1}$. 
Then for every $x\in Z_{\Delta}$ ($Z$ as in a Theorem \ref{key}) the relatively compact set $\{P_x^{\beta}\mid \beta\in\mathbb{Q}_+^*\}$ 
has a unique limit $P_x$ for $\beta\uparrow\infty$, and the process $(\Omega_{\overline{E}},(X_t)_{t\geq 0}, (P_x)_{x\in Z_{\Delta}})$ is a Markov process with the transition semigroup 
$(P_t)_{t\geq 0}$ determined by (\ref{atd}). Moreover, 
\begin{equation}\label{tmv}
P_x[X_t \in Z_{\Delta},X_{t-}\in Z_{\Delta}\ for\ all\ t\geq 0]=1 
\end{equation}
for all $x\in Z_{\Delta}$. The proof of this is again the same as in Theorem 4.3 of Ref. \cite{mrz}.
\centerline{}
Up to this end let $(P_x)_{x\in Z_{\Delta}}$ be as in (\ref{tmv}), and $\Omega$ and $Z_{\Delta}$ be specified by Theorem \ref{key}. Since $P_x[\Omega]=1$ for all 
$x\in Z_{\Delta}$, we may restrict $P_x$ and the coordinate process $(X_t)_{t\geq 0}$ to $\Omega$. Let $({\cal{F}}_t)_{t\geq 0}$ be the natural filtration of $(X_t)_{t\geq 0}$.
Then exactly as in Theorem 4.4. of Ref. \cite{mrz} one shows that
$\mathbb{M}^{Z}:=(\Omega, (X_t)_{t\geq 0}, ({\cal{F}}_t)_{t\geq 0},(P_x)_{x\in Z_{\Delta}})$ is a Hunt process with respect to both the $\rho$-topology and the original topology.
\begin{rem}\label{fin}
Let $\mathbb{M}$ be the trivial extension  of $\mathbb{M}^Z$ to $E_{\Delta}$  (see IV. (3.48) in Ref. \cite{mr}, or IV. (2.18) in Ref. \cite{St1}). Then 
$\mathbb{M}$ is again a Hunt process and  strictly 
properly associated in the resolvent sense with $({\cal{E}},{\cal{F}})$ by (R2) and (\ref{atd}). The Hunt process $\mathbb{M}$ is unique up to the equivalence 
described in IV.6.3 of Ref. \cite{mr}. In this sense $\mathbb{M}$ is the same process as the one constructed in Theorem 3 of Ref. \cite{Tr5} under the condition {\bf SD3}.
\end{rem}
\begin{theo}\label{hunt2}
Let ${\cal{E}}$ be a strictly quasi-regular generalized Dirichlet form satisfying {\bf D3}. Then there exists a 
strictly $m$-tight Hunt process which is strictly properly associated in the resolvent sense with ${\cal{E}}$. 
\end{theo}
\begin{proof}
For the existence of the Hunt process which is strictly  properly associated in the resolvent sense with ${\cal{E}}$ see Remark \ref{fin}. 
The $m$-tightness is a direct consequence of the existence of a strict ${\cal{E}}$-nest like in Definition 2(i) of Ref. \cite{Tr5} and the representation of the capacity 
from Lemma 1(i) in Ref. \cite{Tr5}. 
\end{proof}

\end{document}